\documentclass[11pt]{article}

\usepackage{amssymb}
\usepackage{amsmath}
\usepackage{epsfig}
\usepackage{pstricks}

\def\eqref#1{(\ref{#1})}
\def\dsp{\displaystyle}
\def\Frac#1#2{\frac
{
 {\raise.6ex
 \hbox{$\displaystyle#1$}}
}
{
 {\lower.6ex
 \hbox{$\displaystyle#2$}}
 }
}

\numberwithin{equation}{section}

\newtheorem{theorem}{Theorem}[section]

\newtheorem{lemma}[theorem]{Lemma}

\newtheorem{corollary}[theorem]{Corollary}

\newtheorem{example}[theorem]{Example}
\newtheorem{remark}[theorem]{Remark}

\usepackage{xcolor}

\def\eoproof{{\unskip\nobreak\hfil\penalty50	
\hskip2em\hbox{}\nobreak\hfil$\square$
\parfillskip=0pt\finalhyphendemerits=0\medbreak}}

\def\eps{\varepsilon}

\makeatother

\def\dsp{\displaystyle}
\def\Frac#1#2{\frac
{
 {\raise.6ex
 \hbox{$\displaystyle#1$}}
}
{
 {\lower.6ex
 \hbox{$\displaystyle#2$}}
 }
}

\def\sign{{\rm sign}}

\input epsf

\def\CHF#1#2#3{
{}_1F_1\left(
\begin{array}{c}
\begin{array}{cc} \hskip-10pt#1 \end{array}\\
\begin{array}{c}  \hskip-10pt#2 \end{array}
\end{array}
\hskip-8pt;\,#3
\right)}

\def\CHFs#1#2#3{
{}_1F_1\left({a};{c};{z}\right)
}

\def\binom#1#2{
\renewcommand{\arraystretch}{0.9}
\left(
\begin{array}{c}
\begin{array}{c}\hskip-10pt#1\end{array}\\
\begin{array}{c}\hskip-10pt#2\end{array}
\end{array}
\hskip-10pt
\renewcommand{\arraystretch}{1.0}
\right)}

\def\erfc{{\rm erfc}}

\def\bigO{{\cal O}}

\def\tfrac#1#2{{{\lower.6ex
\hbox{$\scriptstyle#1$}}\over 
{\raise.7ex
\hbox{$\scriptstyle#2$}}}}

\def\RR{\mathbb R}

\def\sign{{\rm sign}}

\def\protectbold#1{\protect{\boldmath{$#1$}}}

\def\erf{{{\rm erf}}}
\def\erfc{{{\rm erfc}}}

\def\erfc{{\rm erfc}}

\def\tfrac#1#2{{{\lower.6ex
\hbox{$\scriptstyle#1$}}\over 
{\raise.7ex
\hbox{$\scriptstyle#2$}}}}

 \title{New asymptotic representations of the \\
 noncentral $t$-distribution}
\author{
Amparo Gil\footnotemark[1]
 \and
Javier Segura\footnotemark[2]
\and
Nico M. Temme\footnotemark[3]
\\
}

\begin{document}

\maketitle

\renewcommand{\thefootnote}{\fnsymbol{footnote}}

\footnotetext[1]{Departamento de Matem\'atica Aplicada y CC. de la Computaci\'on.
ETSI Caminos. Universidad de Cantabria. 39005-Santander, Spain.   }
\footnotetext[2]{ Departamento de Matem\'aticas, Estadistica y 
        Computaci\'on. Universidad de Cantabria, 39005 Santander, Spain.  }
\footnotetext[3]{IAA, 1825 BD 25, Alkmaar, The Netherlands. Former address: Centrum Wiskunde \& Informatica (CWI), Science Park 123, 1098 XG Amsterdam,  The Netherlands. \\ Email: Nico.Temme@cwi.nl }

\begin{abstract}

New asymptotic approximations of the non-central $t$ distribution are given, a generalization of the Student's $t$ distribution. Using new integral representations, we give new asymptotic expansions for large values of the noncentrality parameter but also for large values of the degrees  of freedom parameter. In some case we accept more than one large parameter. These results are in terms of elementary functions, but also in terms of the complementary error function and the incomplete gamma function. A number of numerical tests demonstrate the performance of the asymptotic approximations.

\end{abstract}

\section{Introduction}\label{sec:intro}

The noncentral $t$-distribution has several applications in engineering, biology and other scientific areas; see for example
\cite{Cashman:2007:SMF}, \cite{Lin:2016:SMF}, \cite{MacMullan:1998:SMF}. 
There is a vast literature on the Student's $t$ distribution and its generalizations, including the noncentral distribution. For a recent review, see  \cite{Li:2020:RSD}; and for the particular case of the noncentral $t$ distribution see \cite{Owen:1968:PNT}. 

We introduce the following notation of the noncentral $t$-distribution, which we denote by $F_n(x;\delta)$. Let $x\ge 0$, $\delta\in\RR$, and $n>0$. Then we define
\begin{equation}\label{eq:intro01}
\begin{array}{@{}r@{\;}c@{\;}l@{}}
F_n(x;\delta)&=&\dsp{\tfrac12\erfc\left(\delta/\sqrt2\right)+P_n(x;\delta)+Q_n(x;\delta),}\\[8pt]
P_n(x;\delta)&=&\dsp{\tfrac12 e^{-\frac12\delta^2}\sum_{j=0}^\infty p_j(\delta) I_y\left(j+\tfrac12,\tfrac12 n\right),}\\[8pt]
Q_n(x;\delta)&=&\dsp{\tfrac12 e^{-\frac12\delta^2}\sum_{j=0}^\infty q_j(\delta) I_y\left(j+1,\tfrac12 n\right),}\\[8pt]
p_j(\delta)&=&\dsp{\frac{\left(\frac12\delta^2\right)^j}{j!},\quad q_j(\delta)=\frac{\delta}{\sqrt2}\frac{\left(\frac12\delta^2\right)^j}{\Gamma\left(j+\frac32\right)},
\quad y=\frac{x^2}{n+x^2}.}
\end{array}
\end{equation}
For $x\le 0$ we define
\begin{equation}\label{eq:intro02}
F_n(x;\delta)=1-F_n(-x;-\delta).
\end{equation}

In the first line of \eqref{eq:intro01}  $\erfc(x)$ is the complementary error function defined by
 \begin{equation}\label{eq:intro03}
\erfc(x)=\frac{2}{\sqrt{\pi}}\int_x^\infty e^{-t^2}\,dt=2\Phi\left(-x\sqrt{2}\right),\quad \Phi(x)=\frac{1}{\sqrt{2\pi}}\int_{-\infty}^x e^{-\frac12 t^2}\,dt.
\end{equation}
The function $\Phi(x)$ is called  the normal distribution and   $I_y(a,b)$ is the incomplete beta function defined by 
\begin{equation}\label{eq:intro04}
I_y(a,b)=\frac{1}{B(a,b)}\int_0^y t^{a-1}(1-t)^{b-1}\,dt, \quad B(a,b)=\frac{\Gamma(a)\Gamma(b)}{\Gamma(a+b)}.
\end{equation}

The parameter $\delta$ is the noncentrality parameter. We assume that  $n>0$, not necessarily an integer in this paper; $n$  denotes the degrees of freedom. In the literature the definition of $F_n(x;\delta)$ in the first line of \eqref{eq:intro01} is also given with 
$\frac12\erfc\left(\delta/\sqrt2\right)$ replaced by $\Phi(-\delta)$.

When $\delta=0$, the function $F_n(x;\delta)$ reduces to Student's $t$ distribution $F_n(x)$, which we discussed in our recent paper \cite{Gil:2022:NST}. The definition in \eqref{eq:intro01} corresponds with the one in \cite{Witkovsky:2013:NOC}. In \cite[Eqn.~(5)]{Posten:1994:ANA} the following form (with different notation of parameters) is given
\begin{equation}\label{eq:intro05}
\begin{array}{@{}r@{\;}c@{\;}l@{}}
F_n(x;\delta)&=&1-G_n(x;\delta),\\[8pt]
G_n(x;\delta)&=&\dsp{\tfrac12 e^{-\frac12\delta^2}\sum_{j=0}^\infty \frac{(\delta/\sqrt2)^j}{\Gamma\left(\frac12 j+1\right)}I_{1-y}\left(\tfrac12 n, \tfrac12+\tfrac12j\right),}
\end{array}
\end{equation}
with $y$ as in \eqref{eq:intro01}. Hence, $F_n(x;\delta)$ and $G_n(x;\delta)$ are complementary functions, and in numerical computations it is convenient, to avoid numerical cancellation when one of the two functions is needed, to compute first the {\em primary} function, that is,  $\min\{F_n(x;\delta),G_n(x;\delta)\}$. This is also important for $x<0$ where $F_n(x;\delta)$ may be very small; see Figure~\ref{fig:fig01}, right. In that case the relation in \eqref{eq:intro02} should be written as $F_n(x;\delta)=G_n(-x;-\delta)$, $x<0$.

\begin{figure}[tb]
\vspace*{0.0cm}
\begin{center}
\begin{minipage}{6cm}
        \includegraphics[width=6cm]{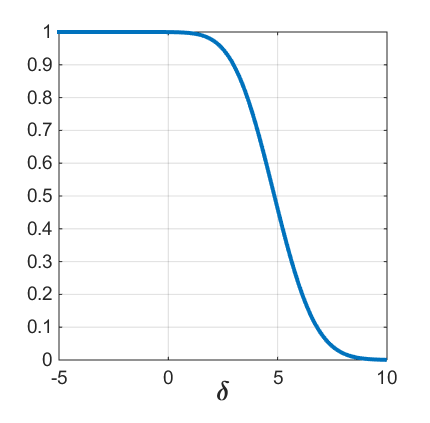} 
\end{minipage}
\hspace*{1.5cm}
\begin{minipage}{6cm}
        \includegraphics[width=5.8cm]{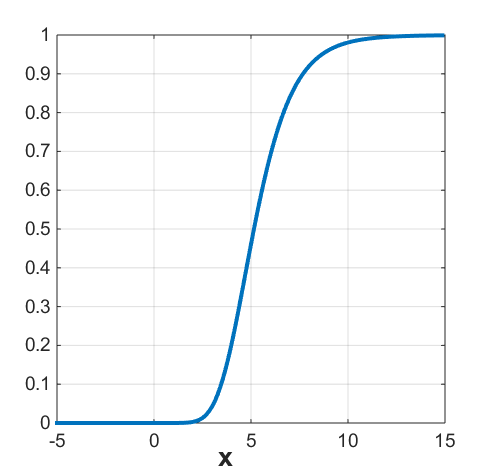} 
  \end{minipage}
\end{center}
\caption{Graphs of $F_n(x;\delta)$ as a function of the parameters $\delta$ and $x$.
{\bf Left:} $ x= 5$, $n= 10$, $-5\le \delta\le 10$.
{\bf Right:}  $-5\le x\le 15$, $n= 10$, $\delta= 5$.
}
\label{fig:fig01}
\end{figure}

In the graphs of Figure~\ref{fig:fig01} we see that the function values are about $\frac12$ when $x=\delta$. We call $x$ (or $\delta$) the {\em transition value}  of $F_n(x;\delta)$ considered as a function of $\delta$ (or of $x$).  When we have derived asymptotic approximations, we will see in different ways that the transition does indeed take place when $x=\delta$.

The representation of $F_n(x;\delta)$ in \eqref{eq:intro05} can be transformed into those in \eqref{eq:intro01} by using the complementary relation of the incomplete beta function
\begin{equation}\label{eq:intro06}
I_y(a,b)=1-I_{1-y}(b,a),
\end{equation}
and the Kummer function, or confluent hypergeometric function, defined by
\begin{equation}\label{eq:intro07}
\CHF{a}{b}{z}=\sum_{k=0}^\infty \frac{(a)_n}{(b)_n}\frac{z^n}{n!},\quad (a)_n=\frac{\Gamma(a+n)}{\Gamma(a)},
\end{equation}
with the special case
\begin{equation}\label{eq:intro08}
\CHF{1}{\frac32}{z^2}=\frac{\sqrt{\pi}}{2z} e^{z^2}\erf\,z.
\end{equation}

\begin{remark}\label{rem:rem01}
In \cite[Eqn.~(31.17)]{Johnson:1995:CUD} the form \eqref{eq:intro01} is given without the first term $\frac12\erfc\left(\delta/\sqrt2\right)$.
The series expansion of $2P_n(x;\delta)$ that follows from \eqref{eq:intro01} is given in  \cite[Page~455]{Amos:1964:RBD} as one of the many definitions of the noncentral $t$-distribution that do not correspond with our definition in \eqref{eq:intro01}. 
\end{remark}

\begin{remark}\label{rem:rem02}
The function $P_n(x;\delta)$ introduced in \eqref{eq:intro01} is a special case of the noncentral  beta distribution that we considered in a recent paper \cite{Gil:2019:NCB} in  the form
\begin{equation}\label{eq:intro09}
B_{p,q}(x,y)=e^{-x/2} \displaystyle\sum_{j=0}^{\infty} \Frac{1}{j!}\left(\Frac{x}{2}\right)^j I_y (p+j,q),\quad 0<y<1.
\end{equation}
We have
\begin{equation}\label{eq:intro10}
P_n(x;\delta)=\tfrac12B_{\frac12,\frac12n}\left(\delta^2,y\right),\quad y=\frac{x^2}{n+x^2}.
\end{equation}
In the previous paper we have  derived recurrence relations of  $B_{p,q}(x,y)$ with respect to $p$ and to $q$, which cannot be used for $P_n(x;\delta)$ with respect to $n$, because $y$ in \eqref{eq:intro10} depends on $n$ as well.
\end{remark}

In the first sections we give new integral representations  for $F_n(x;\delta)$ and for the auxiliary functions $P_n(x;\delta)$ and $Q_n(x;\delta)$. In later sections, using these integral representations, we give new asymptotic expansions for large values of the noncentrality parameter $\delta$ but also for large values of the degrees  of freedom parameter $n$. In several cases we accept more than one large parameter. These results are in terms of elementary functions, but also in terms of the complementary error function and the incomplete gamma function. A number of numerical tests demonstrate the performance of the asymptotic approximations.

\newpage

\section{Real integral representations}\label{sec:realint}

We have the following lemma.
\begin{lemma}\label{lem:lem01}
Let $y=x^2/(n+x^2)$, then
\begin{equation}\label{eq:realint01}
\begin{array}{@{}r@{\;}c@{\;}l@{}}
P_n(x;\delta)&=&\dsp{\tfrac12 \frac{e^{-\frac12\delta^2}}{B(\frac12,\frac12n)}
\int_0^yt^{-\frac12}(1-t)^{\frac12n-1}\CHF{\frac12n+\frac12}{\frac12}{\tfrac12\delta^2t}\,dt}\\[8pt]
&=&\frac12-\dsp{\tfrac12 \frac{e^{-\frac12\delta^2}}{B(\frac12,\frac12n)}
\int_0^{1-y}t^{\frac12n-1}(1-t)^{-\frac12}\CHF{\frac12n+\frac12}{\frac12}{\tfrac12\delta^2(1-t)}\,dt,}\\[8pt]
Q_n(x;\delta)&=&\dsp{\tfrac12 \frac{\delta n\,e^{-\frac12\delta^2}}{\sqrt{2\pi}}\int_0^y(1-t)^{\frac12n-1}\CHF{\frac12n+1}{\frac32}{\tfrac12\delta^2t}\,dt}\\[8pt]
&=&\frac12\erf\left(\delta/\sqrt2\right)-\dsp{\frac{\delta n\,e^{-\frac12\delta^2}}{2\sqrt{2\pi}}\int_0^{1-y}t^{\frac12n-1}\CHF{\frac12n+1}{\frac32}{\tfrac12\delta^2(1-t)}\,dt.}
\end{array}
\end{equation}
\end{lemma}
\begin{proof}
The proof follows from \eqref{eq:intro01} by using  the integral representation \eqref{eq:intro04} and the series expansion \eqref{eq:intro07}. 
For the  integrals with intervals of integration $[0,1-y]$ we first change the integrals over $[0,y]$ by writing $[0,y]=[0,1]\setminus[y,1]$ and then we use the change of variable $t\to1-t$ for the integral over $[y,1]$. The complete integral over $[0,1]$ becomes for the first case, with $z=\frac12\delta^2$,
\begin{equation}\label{eq:realint02}
\begin{array}{@{}r@{\;}c@{\;}l@{}}
&&\dsp{\tfrac12 \frac{e^{z}}{B(\frac12,\frac12n)}
\int_0^1t^{-\frac12}(1-t)^{\frac12n-1}\CHF{\frac12n+\frac12}{\frac12}{zt}\,dt}\\[8pt]
&=&\dsp{\tfrac12 \frac{e^{-z}}{B(\frac12,\frac12n)}  
\sum_{j=0}^\infty\frac{z^j}{j!}\frac{\left(\frac12n+\frac12\right)_j}{\left(\frac12\right)_j}
\int_0^1 t^{j-\frac12}(1-t)^{\frac12n-1}}\,dt\\[8pt]
&=&\dsp{\tfrac12 \frac{e^{-z}}{B(\frac12,\frac12n)}  \frac{\Gamma\left(\frac12n\right) \Gamma\left(\frac12\right)}{\Gamma\left(\frac12n+\frac12\right)}
\sum_{j=0}^\infty\frac{z^j}{j!}=\tfrac12.
}\\[8pt]
\end{array}
\end{equation}
The second case follows similarly by using  \eqref{eq:intro08}.
\eoproof
\end{proof}
The integral of $P_n(x;\delta)$  in \eqref{eq:realint01} also corresponds with a more general integral  of the noncentral beta distribution; see  \cite[Eqn.~(1.6)]{Gil:2019:NCB}. The Kummer functions in these integrals are related with parabolic cylinder functions, see \cite[Section~12.7(iv)]{Temme:2010:PCF}.

A known integral representation is given in the following lemma.
\begin{lemma}\label{lem:lem02}
Let $\delta\in\RR$, $x\ge0$ and $n>0$, then
\begin{equation}\label{eq:realint03}
F_n(x;\delta)=A_n \int_0^\infty
\erfc\left((\delta-xt)/\sqrt{2}\right) e^{-\frac12nt^2}t^{n-1}\,dt,\quad  A_n=\frac{(n/2)^{n/2}}{\Gamma(n/2)}.
\end{equation}
\end{lemma}
\begin{proof}This integral is given, with different notation, in \cite[Section~13]{Owen:1968:PNT} and it also follows from the
 the integral representation derived in  \cite[Eqn.~(10)]{Witkovsky:2013:NOC}:
\begin{equation}\label{eq:realint04}
F_n(x;\delta)=\tfrac12\erfc\left(\delta/\sqrt{2}\right)+\frac{1}{\sqrt{2\pi}}\int_0^\infty e^{-\frac12(s-\delta)^2}Q\left(\tfrac12n,\tfrac12ns^2/x^2\right)\,ds.
\end{equation}
The function $Q(a,z)$ is  the normalised incomplete gamma functions, which together with the complementary function $P(a,z)=1-Q(a,z)$ is defined by
\begin{equation}\label{eq:realint05}
P(a,z)=\frac{1}{\Gamma(a)}\int_0^z t^{a-1} e^{-t}\,dt, \quad Q(a,z)=\frac{1}{\Gamma(a)}\int_z^\infty t^{a-1} e^{-t}\,dt,
 \end{equation}
with $\Re a >0$ for the function $P(a,z)$. Because
\begin{equation}\label{eq:realint06}
\frac{1}{\sqrt{2\pi}} e^{-\frac12(s-\delta)^2}=-\tfrac12\frac{d}{ds}\erfc\left((s-\delta)/\sqrt{2}\right),
 \end{equation}
we can integrate \eqref{eq:realint04} by parts, and obtain the representation
\begin{equation}\label{eq:realint07}
F_n(x;\delta)=\left(\frac{n}{2x^2}\right)^{n/2}\frac{1}{\Gamma(n/2)}
\int_0^\infty \erfc\left((\delta-s)/\sqrt{2}\right) s^{n-1} e^{-ns^2/(2x^2)}\,ds.
\end{equation}
We have used the relations
\begin{equation}\label{eq:realint08}
\erfc(z)=2-\erfc(-z),\quad Q(a,0)=1, \quad  \frac{d}{dz}Q(a,z)=-\frac{z^{a-1}e^{-z}}{\Gamma(a)},
\end{equation}
and
\begin{equation}\label{eq:realint09}
\left(\frac{n}{2x^2}\right)^{n/2}\frac{1}{\Gamma(n/2)}
\int_0^\infty s^{n-1} e^{-ns^2/(2x^2)}\,ds=A_n \int_0^\infty
e^{-\frac12nt^2}t^{n-1}\,dt=\tfrac12.
\end{equation}
Substituting in \eqref{eq:realint07} $s=xt$ gives the integral in \eqref{eq:realint03}.
\eoproof
\end{proof}

\begin{corollary}\label{cor:cor01}
We have the alternative representations, again for $x\ge0$,
\begin{equation}\label{eq:realint10}
\begin{array}{@{}r@{\;}c@{\;}l@{}}
F_n(x;\delta)&=&\dsp{1-A_n e^{-\frac12n}\int_0^\infty
\erfc\left((xt-\delta)/\sqrt{2}\right) e^{-\frac12n\phi(t)}\frac{dt}{t},}\\[8pt]
&=&\dsp{\tfrac12-A_ne^{-\frac12n}\int_0^\infty
\erf\left((\delta-xt)/\sqrt{2}\right) e^{-\frac12n\phi(t)}\frac{dt}{t},}
\end{array}
\end{equation}
where
\begin{equation}\label{eq:realint11}
\phi(t)=t^2-\ln(t^2)-1.
\end{equation}
The first relation corresponds with \eqref{eq:intro02}.
\end{corollary}
\begin{proof} This follows from the  relation for $\erfc(z)=2-\erfc(-z)=1-\erf(z)$, and from \eqref{eq:realint09} 
\eoproof
\end{proof}

\begin{corollary}\label{cor:cor02}
Let $y=x^2/(n+x^2)$. We have in terms of parabolic cylinder functions
\begin{equation}\label{eq:realint12}
\begin{array}{@{}r@{\;}c@{\;}l@{}}
\dsp{\frac{\partial}{\partial x}F_n(x;\delta)}&=&\dsp{\sqrt{\frac{2}{\pi}}\,A_n\Gamma(n+1) 
\frac{e^{-\frac14\delta^2(2-y)}}{(n+x^2)^{\frac12n+\frac12}}     
U\left(n+\tfrac12,-\delta\sqrt{y}\right),}\\[8pt]
\dsp{\frac{\partial}{\partial \delta}F_n(x;\delta)}&=&\dsp{-\sqrt{\frac{2}{\pi}}\,A_n \Gamma(n) 
\frac{e^{-\frac14\delta^2(2-y)}}{(n+x^2)^{\frac12n}}     
U\left(n-\tfrac12,-\delta\sqrt{y}\right).}\\[8pt]
\end{array}
\end{equation}
\end{corollary}
\begin{proof} The proof follows from \eqref{eq:intro03},  the integral in \eqref{eq:realint03} and the integral representation of the parabolic cylinder function (see  \cite[Section~12.5(i)]{Temme:2010:PCF})
\begin{equation}\label{eq:realint13}
U(a,z)=\frac{e^{-\frac14z^2}}{\Gamma\left(a+\frac12\right)}\int_0^\infty t^{a-\frac12}e^{-\frac12t^2-zt}\,dt, \quad \Re a>-\tfrac12.
\end{equation}
\eoproof

\end{proof}
The derivatives show clearly the monotonicity of $F_n(x;\delta)$ for $x\ge0$ and $\delta\in\RR$.

\begin{figure}[tb]
\vspace*{0.8cm}
\begin{center}
\begin{minipage}{6cm}
        \includegraphics[width=6cm]{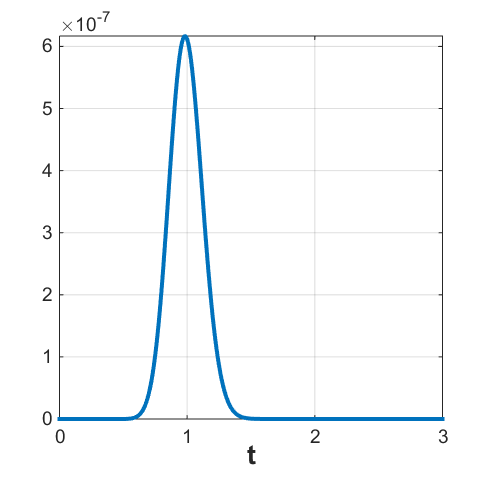} 
\end{minipage}
\hspace*{1.5cm}
\begin{minipage}{6cm}
        \includegraphics[width=6.2cm]{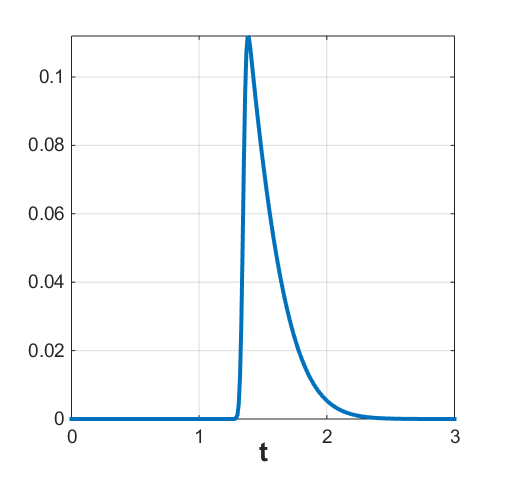} 
  \end{minipage}
\end{center}
\caption{Graphs of the integrands of the integral in \eqref{eq:realint03}.
{\bf Left:} $ x= 20$, $n= 30$, $\delta= 10$.
{\bf Right:}  $ x= 52$, $n= 4$, $\delta= 70$.
}
\label{fig:fig02}
\end{figure}

 The complementary error function in the integral in \eqref{eq:realint03} becomes almost equal to 2 for large positive $t$ and  when $t\ll\delta/x$ this function will become exponentially small.  This gives is a change of behaviour when $t=\delta/x$.  In addition,  the exponential part of the integrand $\exp(-\frac12n \phi(t))$ has its maximal behaviour when $t=1$. Consequently, these special behaviours of two parts of the integrand give a sudden change in behaviour of $F_n(x;\delta)$ when $\delta$ passes the value $x$, in particular when $\delta$ and $x$ are large. In that case $F_n(x;\delta)$ passes the value $\frac12$. 
 
In Figure~\ref{fig:fig02} we give two graphs of the integrands of the integral in \eqref{eq:realint03} for two sets of the parameters. The left graph is nicely bell shaped, the right one, with a larger value of $\delta$, shows a very steep left side near $t=\delta/x\doteq1.35$, where the complementary error function rapidly changes values from $0$ to $2$. Numerical quadrature experiences this as a discontinuity.

When we have algorithms for computing $F_n(x;\delta)$ as well as for  the complementary function $G_n(x;\delta)$ defined in\eqref{eq:intro05}, it is important to compute first the {\em primary} function  $\min\{F_n(x;\delta), G_n(x;\delta)\}$ (see below \eqref{eq:intro05}). This means that if $\delta \ge x$ we should calculate $F_n(x;\delta)$ first, otherwise it is better to calculate $G_n(x;\delta)$. The other function follows from the complementary relation $F_n(x;\delta)+G_n(x;\delta)=1$. Observe that the first line in \eqref{eq:realint10} gives an integral of the  function $G_n(x;\delta)$.

\section{Contour integral representations}\label{sec:contour}
We use the following integral representation of the incomplete beta function (see \cite[\S8.17(iii)]{Paris:2010:INC}):
\begin{equation}\label{eq:contour01}
I_y(p,q)=\frac{y^p(1-y)^q}{2\pi i}\int_{c-i\infty}^{c+i\infty} \frac{dt}{t^p(1-t)^q(t-y)},\quad 0<y<c<1,
\end{equation}
with $p>0$, $q>0$.  The branches of $s^{-a}$ and $(1-s)^{-b}$  are continuous on the path and assume their principal values when $t=c$.

The function $P_n(x;\delta) $ in  \eqref{eq:intro01} becomes
\begin{equation}\label{eq:contour02}
\begin{array}{@{}r@{\;}c@{\;}l@{}}
P_n(x;\delta)&=&\dsp{
\tfrac12 e^{-\frac12\delta^2}\frac{\sqrt{y}\,(1-y)^{\frac12n}}{2\pi i}\int_{c-i\infty}^{c+i\infty}  \frac{e^{\frac{1}{2t}\delta^2y}}{\sqrt{t}\, (1-t)^{\frac12n}(t-y)}\, dt}\\[8pt]
&=&\dsp{\tfrac12-
\tfrac12 e^{-\frac12\delta^2}\frac{\sqrt{y}\,(1-y)^{\frac12n}}{2\pi i}\int_{d-i\infty}^{d+i\infty}  \frac{e^{\frac{1}{2t}\delta^2y}}{\sqrt{t}\, (1-t)^{\frac12n}(y-t)}\, dt,}
\end{array}
\end{equation}
where $0<y<c<1$ and $0<d<y<1$. The argument of the exponential function is a bounded quantity on the path of integration, hence, interchanging summation and integration can be justified because of uniform convergence. The second follows from shifting in the first integral the path to the right, across the pole at $t=y$.

For the function $Q_n(x;\delta)$ we obtain, using  \eqref{eq:intro01}, \eqref{eq:intro07},  \eqref{eq:intro08} and  \eqref{eq:contour01},
\begin{equation}\label{eq:contour03}
\begin{array}{@{}r@{\;}c@{\;}l@{}}
Q_n(x;\delta)&=&\dsp{\tfrac12 \frac{\delta}{\sqrt2} \frac{e^{-\frac12\delta^2}}{\Gamma\left(\frac32\right)}
\frac{y(1-y)^{\frac12n}}{2\pi i}\int_{c-i\infty}^{c+i\infty}  \frac{{}_1F_1\left(1,\frac32;\frac{1}{2t}\delta^2y\right)}{t (1-t)^{\frac12n}(t-y)}\,dt,}\\[8pt]
&=&\dsp{\tfrac12 e^{-\frac12\delta^2}\frac{\sqrt{y}\,(1-y)^{\frac12n}}{2\pi i}
\int_{c-i\infty}^{c+i\infty}  \frac{e^{\frac{1}{2t}\delta^2y}\erf\left(\delta\sqrt{\frac{y}{2t}}\,\right) }{\sqrt{t}\, (1-t)^{\frac12n}(t-y)}\, dt.}
\end{array}
\end{equation}
When we use $\erf\,z=1-\erfc\,z$, we obtain
\begin{equation}\label{eq:contour04}
Q_n(x;\delta)=P_n(x;\delta)-\tfrac12 e^{-\frac12\delta^2}\frac{\sqrt{y}\,(1-y)^{\frac12n}}{2\pi i}\int_{c-i\infty}^{c+i\infty}  \frac{e^{\frac{1}{2t}\delta^2y}\erfc\left(\delta\sqrt{\frac{y}{2t}}\,\right)}{\sqrt{t}\, (1-t)^{\frac12n}(t-y)}\, dt.
\end{equation}
We conclude:
\begin{lemma}\label{lem:lem03}
We have derived the following representation
\begin{equation}\label{eq:contour05}
F_n(x;\delta)=2P_n(x;\delta)+R_n(x;\delta), 
\end{equation}
where
\begin{equation}\label{eq:contour06}
\begin{array}{@{}r@{\;}c@{\;}l@{}}
R_n(x;\delta)&=&\tfrac12\erfc\left(\delta/\sqrt2\right)+Q_n(x;\delta)-P_n(x;\delta)\\[8pt]
&=&\tfrac12\erfc\left(\delta/\sqrt2\right)\ -\\[8pt]
&&\dsp{\tfrac12 e^{-\frac12\delta^2}\frac{\sqrt{y}\,(1-y)^{\frac12n}}{2\pi i}\int_{c-i\infty}^{c+i\infty}  \frac{e^{\frac{1}{2t}\delta^2y}\erfc\left(\delta\sqrt{\frac{y}{2t}}\,\right)}{\sqrt{t}\, (1-t)^{\frac12n}(t-y)}\, dt}\\[8pt]
&=&\dsp{\tfrac12 e^{-\frac12\delta^2}\frac{\sqrt{y}\,(1-y)^{\frac12n}}{2\pi i}\int_{d-i\infty}^{d+i\infty}  \frac{e^{\frac{1}{2t}\delta^2y}\erfc\left(\delta\sqrt{\frac{y}{2t}}\,\right)}{\sqrt{t}\, (1-t)^{\frac12n}(y-t)}\, dt,}
\end{array}
\end{equation}
where  $0<y<c<1$ and $0<d<y<1$.
\end{lemma}

We find,  using $\erfc(-z)=2-\erfc\,z$ and because $P_n(x;\delta)$ is even with respect to $\delta$ and $Q_n(x;\delta)$ is odd:
\begin{equation}\label{eq:contour07}
F_n(x;-\delta)=1-\tfrac12\erfc\left(\delta/\sqrt{2}\right)+P_n(x;\delta)-Q_n(x;\delta),
\end{equation}
and we conclude
\begin{corollary}\label{cor:cor03}
With $R_n(x;\delta)$ given in \eqref{eq:contour06}, the following representation holds
\begin{equation}\label{eq:contour08}
F_n(x;-\delta)=1-R_n(x;\delta),\quad G_n(x;-\delta)=R_n(x;\delta),
\end{equation}
where $ G_n(x;\delta)$ is the complementary function defined in  \eqref{eq:intro05}. 
\end{corollary}

When we compare this with the representation in \eqref{eq:intro05}, we see that
\begin{equation}\label{eq:contour09}
R_n(x;\delta)=\tfrac12 e^{-\frac12\delta^2}\sum_{j=0}^\infty \frac{(-\delta/\sqrt2)^j}{\Gamma\left(\frac12 j+1\right)}I_{1-y}\left(\tfrac12 n, \tfrac12+\tfrac12j\right). 
\end{equation}
This also follows by using in \eqref{eq:contour06} and the expansion \cite[\S7.2(i), \S6(i)]{
Temme:2010:ERF}
\begin{equation}\label{eq:contour10}
w(z)=e^{-z^2}\erfc(-iz),\quad w(z)=\sum_{k=0}^\infty \frac{(iz)^k}{\Gamma\left(\frac12k+1\right)},
\end{equation}
and the integral representation of the incomplete beta function in \eqref{eq:contour01}.

The next theorem will show that the second  integral in  \eqref{eq:contour06} can be written as a simple Laplace integral. 
\begin{theorem}\label{them:them01}
The function $R_n(x;\delta)$ has the following representation
\begin{equation}\label{eq:contour11}
R_n(x;\delta)=\frac{e^{-\frac12\delta^2}\sqrt{y}\,(1-y)^{\frac12n}}{2\pi}\int_0^\infty
e^{-\frac12\delta^2y t}\frac{t^{\frac12n-\frac12}}{(t+1)^{\frac12n}(1+yt)}\,dt.
\end{equation}
\end{theorem}
\begin{proof}
The proof follows from shifting the contour of the second integral in \eqref{eq:contour06} to the left,  along the imaginary axis. The contour can be further modified by bending it along the negative axis, taking into account the multivalued function $\sqrt{t}$.  Convergence at the origin follows from the asymptotic behaviour of the complementary error function (see  \cite[7.12.1]{Temme:2010:ERF} for more details)
\begin{equation}\label{eq:contour12}
\erfc\,z=\frac{e^{-z^2}}{z\sqrt\pi}\left(1+\bigO\left(z^{-2}\right)\right),\quad z\to \infty.
\end{equation}
The relation $\erfc(iw)+\erfc(-iw)=2$ and a few extra steps finish the proof.
\eoproof
\end{proof}

In \cite[\S5.2]{Gil:2019:NCB} we have used a finite  loop integral of the noncentral beta distribution, which function is  related with $P_n(x;\delta)$, see \eqref{eq:intro10}. In the present case we have the following theorem.
\begin{theorem}\label{them:them02}
The functions $P_n(x;\delta)$, $Q_n(x;\delta)$ and $R_n(x;\delta)$  have the following integral representations
\begin{equation}\label{eq:contour13}
\begin{array}{@{}r@{\;}c@{\;}l@{}}
P_n(x;\delta)&=&\dsp{\tfrac12e^{-\frac12\delta^2}\frac{\sqrt{y}\,(1-y)^{\frac12n}}{2\pi i}\int_0^{(1+)}e^{\frac12\delta^2 y\tau}\,\frac{\tau^{\frac12n-\frac12}}{(\tau-1)^{\frac12n}(1-y\tau)}\,d\tau,}\\[8pt]
Q_n(x;\delta)&=&\dsp{\tfrac12e^{-\frac12\delta^2}\frac{\sqrt{y}\,(1-y)^{\frac12n}}{2\pi i}\int_0^{(1+)}e^{\frac12\delta^2 y\tau}\,\frac{\tau^{\frac12n-\frac12}\erf\left(\delta\sqrt{\frac{1}{2}y\tau}\,\right)}{(\tau-1)^{\frac12n}(1-y\tau)}\,d\tau,}\\[8pt]
R_n(x;\delta)&=&\tfrac12\erfc\left(\delta/\sqrt2\right)\ -\\[8pt]
&&\dsp{\tfrac12e^{-\frac12\delta^2}\frac{\sqrt{y}\,(1-y)^{\frac12n}}{2\pi i}\int_0^{(1+)}e^{\frac12\delta^2 y\tau}\,\frac{\tau^{\frac12n-\frac12}\erfc\left(\delta\sqrt{\frac{1}{2}y\tau}\,\right)}{(\tau-1)^{\frac12n}(1-y\tau)}\,d\tau.}
\end{array}
\end{equation}
The finite contours  start at the origin, encircle the point $\tau=1$ anti-clockwise, and return to the origin. The multivalued factors assume their principal values, $(\tau-1)$ is positive for $\tau>1$.  The pole at $\tau=1/y$ is outside the contours, which means,  the contours cut the positive real axis between $1$ and $1/y$.
\end{theorem}
\begin{proof}
These representations immediately follow from  \eqref{eq:contour02}, \eqref{eq:contour03} and \eqref{eq:contour06} by using the substitution $t=1/\tau$. \eoproof
\end{proof}

From this theorem a similar result for $F_n(x;\delta)$ easily follows. Observe that the second representation in the following theorem gives a result for the complementary function $G_n(x;\delta)$ introduced in \eqref{eq:intro05}. We have
\begin{theorem}\label{them:them03}
The noncentral $t$ distribution $F_n(x;\delta)$ defined in \eqref{eq:intro01} has the integral representations
\begin{equation}\label{eq:contour14}
\begin{array}{@{}r@{\;}c@{\;}l@{}}
F_n(x;\delta)&=&\dsp{\tfrac12\erfc\left(\delta/\sqrt2\right)\ +}\\[8pt]
&&\dsp{\tfrac12e^{-\frac12\delta^2}\frac{\sqrt{y}\,(1-y)^{\frac12n}}{2\pi i}
\int_0^{(1+)}e^{\frac12\delta^2 y\tau}\frac{\tau^{\frac12n-\frac12}
\erfc\left(-\delta\sqrt{\frac{1}{2}y\tau}\,\right)}{(\tau-1)^{\frac12n}(1-y\tau)}\,d\tau,}\\[8pt]
&=&\dsp{1-\tfrac12e^{-\frac12\delta^2}\frac{\sqrt{y}\,(1-y)^{\frac12n}}{2\pi i}
\int_0^{(1+,1/y+)}e^{\frac12\delta^2 y\tau}\frac{\tau^{\frac12n-\frac12}
\erfc\left(-\delta\sqrt{\frac{1}{2}y\tau}\,\right)}{(\tau-1)^{\frac12n}(y\tau-1)}\,d\tau,}
\end{array}
\end{equation}
where the first contour is as in \eqref{eq:contour13} and the second one starts  at the origin, encircles the point $\tau=1$ and $\tau=1/y$ anti-clockwise, and returns to the origin. 
\end{theorem}
\begin{proof}
The sum $P_n(x;\delta)+Q_n(x;\delta)$ with integral representations in \eqref{eq:contour13} gives a similar integral with $1+\erf(\delta\sqrt{y\tau/2})=1-\erf(-\delta\sqrt{y\tau/2})=\erfc(-\delta\sqrt{y\tau/2})$. By using this in \eqref{eq:intro01} the first representation follows. For the second representation we modify  the contour in the first representation by taking the pole at $\tau=1/y$ inside the contour and picking up the residue. We use 
 the property $\erfc(z)+\erfc(-z)=2$ to conclude the proof. 
\eoproof
\end{proof}
 
The integral representation of $R_n(x;\delta)$ given in \eqref{eq:contour11} of  Theorem~\ref {them:them01} plays also a role in the following theorem.
\begin{theorem}\label{them:them04}
Let the function $H_n(x;\delta)$ be defined by the contour integral
\begin{equation}\label{eq:contour15}
H_n(x;\delta)=\tfrac12e^{-\frac12\delta^2}\frac{\sqrt{y}\,(1-y)^{\frac12n}}{2\pi i}\int_{-\infty}^{(1+)}
e^{\frac12\delta^2 y\tau}\,\frac{\tau^{\frac12n-\frac12}}{(\tau-1)^{\frac12n}(1-y\tau)}\,d\tau,
\end{equation}
with $y=x^2/(n+x^2)$. The contour starts at $-\infty$ with the phases of $\tau$ and $(\tau-1)$ equal to $-\pi$, encircles $\tau=1$ anti-clockwise, and terminates at $-\infty$ with the phases of $\tau$ and $(\tau-1)$ equal to $+\pi$. The contour cuts the positives axis between $1$ and $1/y$. Then,
\begin{equation}\label{eq:contour16}
H_n(x;\delta)=P_n(x;\delta)+R_n(x;\delta),\quad F_n(x;\delta)=2H_n(x;\delta)-R_n(x;\delta).
\end{equation}
 \end{theorem}
 \begin{proof}
We take the parts of the contour with $\Re\tau <0$ below and above the negative axis, taking into account the phase of $\tau$ and $\tau-1$. The remaining  integral with path $\int_0^{(1+)}$  gives the representation in \eqref{eq:contour13} of  $P_n(x;\delta)$, and the integrals along $(-\infty,0)$ give together the representation of $R_n(x;\delta)$ in  \eqref{eq:contour11}. This gives the first relation in \eqref{eq:contour16}, and the second one follows from \eqref{eq:contour05}.
\eoproof
\end{proof}

\begin{remark}\label{rem:rem03}
The integral in \eqref{eq:contour11}  becomes a Kummer $U$ function when $(1+yt)$ is removed. Similarly for the first integral in \eqref{eq:contour13}  and the integral in \eqref{eq:contour15}, without $(1-y\tau)$, the integrals become Kummer $F$ functions. The relation in \eqref{eq:contour16} corresponds with the relation in \cite[Eq.~13.2.42]{Olde:2010:CHF}, and the proof of this relation for the Kummer functions can be based on the integral representations in 
\cite[Eqns.~13.4.4,  13.4.9, 13.4.13]{Olde:2010:CHF}.
\end{remark}

\section{Preliminary observations for large values of~\protectbold{\delta}}\label{sec:observe}
We will show in this section that we only need to consider the role of $P_n(x;\delta)$ in the asymptotic analysis of  $F_n(x;\delta)$ for large values of $\delta$.
We explain this in the simple example where we start with the even and odd power series of $e^x$:
\begin{equation}\label{eq:observe01}
e^x=\sum_{k=0}^\infty\frac{x^{2k}}{(2k)!}+\sum_{k=0}^\infty\frac{x^{2k+1}}{(2k+1)!}=2\cosh x+\left(\sinh x-\cosh  x\right)=2\cosh x-e^{-x}.
\end{equation}
When we want to compute  $e^x$ for large values of $x$ by using the Taylor series,  it is possible to use twice the expansion of $\cosh x$, with a relative  error term of order~$e^{-2x}$.

In the present case we have the following theorem.
\begin{theorem}\label{them:them05}
The noncentral $t$ distribution $F_n(x;\delta)$ defined in \eqref{eq:intro01} has for large values of the noncentrality parameter $\delta$ the asymptotic representation 
\begin{equation}\label{eq:observe02}
F_n(x;\delta)=2P_n(x;\delta)+\frac{e^{-\frac12\delta^2}\sqrt{y}\,(1-y)^{\frac12n}}{2\pi}\frac{\Gamma\left(\frac12n+\frac12\right)}{\zeta^{\frac12n+\frac12}}\left(1+\bigO\left(\delta^{-2}\right)\right),
\end{equation}
where $\zeta=\frac12\delta^2y$, and  $y\in(0,1)$ and $n$ are fixed.
\end{theorem}

\begin{proof}
The proof follows from \eqref{eq:contour05}, the representation of $R_n(x;\delta)$ in \eqref{eq:contour11} and a straightforward application of Watson's lemma; see  \cite[Chapter~2]{Temme:2015:AMI}.
\eoproof
\end{proof}

When we compute  $F_n(x;\delta)$ by using the relation in \eqref{eq:contour05} it is needed to compute $R_n(x;\delta)$. However, in the transition case $\delta\sim x$, where  $F_n(x;\delta)\sim\frac12$, $R_n(x;\delta)$ can be neglected when $\delta$ is large enough. This follows from \eqref{eq:observe02}.      More details on the asymptotic expansion of $R_n(x;\delta)$ will be given in Theorem~\ref{them:them10} of Section~\ref{sec:unrestrict}, where we shall see that there is an extra exponentially small behaviour noticeable for large values of $n$ and $\delta$.

\section{Asymptotic expansions for large \protectbold{\delta}}\label{sec:delta}
From Theorem~\ref{them:them05} it follows that we can concentrate on the asymptotic behaviour of $P_n(x;\delta)$. In the first part of this section we summarize the corresponding result derived for the noncentral beta distribution in \cite[\S6.1]{Gil:2019:NCB}. In Subsection~\ref{sec:xdelta} we give a new expansion which is not given in  our earlier paper, and which can be modified with respect to the parameters (see \eqref{eq:intro10})  to obtain a similar result for  the noncentral beta distribution.

\subsection{An expansion in terms of elementary functions} \label{sec:summncbeta}

The result of this subsection is a modification with respect to notation of a result  for the noncentral beta distribution as derived in \cite[\S6.1]{Gil:2019:NCB}.

\begin{theorem}\label{them:them06}
The function $F_n(x;\delta)$ defined in  \eqref{eq:intro01} has the asymptotic expansion
\begin{equation}\label{eq:delta01} 
F_n(x;\delta)\sim\frac{e^{-\frac12(1-y)\delta^2}\sqrt{y} (1-y)^{\frac12n} \zeta^{\frac12n-1}}{\Gamma\left(\frac12n\right)}\sum_{k=0}^\infty(-1)^k\left(1-\tfrac12n\right)_k\frac{c_k}{\zeta^k},
\end{equation}
as $ \zeta\to\infty$, where $\zeta=\frac12\delta^2y$. The coefficients $c_k$ follow from the expansion
\begin{equation}\label{eq:delta02}
\frac{(1+t)^{\frac12n-\frac12}}{1-y(t+1)}=\sum_{n=0}^\infty c_k t^k.
\end{equation}
The expansion in \eqref{eq:delta01} can be used for $n={\cal{O}}(1)$ and for $\eps \le y \le 1-\eps$, where $\eps$ is a fixed positive small number. When $n=2m$, $m=1,2,3,\dots$  the terms in  the expansion with index $k\ge m$ all vanish. 
\end{theorem}

The coefficients $c_k$ defined in \eqref{eq:delta02} follow from the simple recursions
\begin{equation}\label{eq:delta03}
c_{k}=\frac{1}{1-y}\left(a_{k}+y c_{k-1}\right), \quad k\ge 1 ,\quad c_0=\frac{1}{1-y}, 
\end{equation}
where the coefficients $a_k$ follow from
\begin{equation}\label{eq:delta04}
(1+t)^{\frac12n-\frac12}=\sum_{k=0}^\infty a_k t^k,\quad a_k=\binom{\frac12n-\frac12}{k}=\frac{\frac12n+\frac12-k}{k}a_{k-1},
\end{equation}
for  $k\ge1$ with $a_0=1$.

\begin{proof}
We start with an expansion of  $H_n(x;\delta)$, and the proof is based on the representation in \eqref{eq:contour15} after the transformation $\tau=1+t$. This gives
\begin{equation}\label{eq:delta05}
H_n(x;\delta)=\tfrac12e^{-\frac12(1-y)\delta^2}\frac{\sqrt{y}\,(1-y)^{\frac12n}}{2\pi i}\int_{-\infty}^{(0+)}e^{\zeta t}t^{-\frac12n} \frac{(1+t)^{\frac12n-\frac12}}{1-y(1+t)}\,dt.
\end{equation}
We apply Watson's lemma for loop integrals  (see \cite[Page~120]{Olver:1997:ASF}) by substituting the expansion \eqref{eq:delta02}, and using the contour integral
\begin{equation}\label{eq:delta06}
\frac{1}{\Gamma(z)}=\frac{1}{2\pi i} \int_{-\infty}^{(0+)}e^t t^{-z}\,dt.
\end{equation}
We obtain
\begin{equation}\label{eq:delta07} 
H_n(x;\delta)\sim\tfrac12\frac{e^{-\frac12(1-y)\delta^2}\sqrt{y} (1-y)^{\frac12n} \zeta^{\frac12n-1}}{\Gamma\left(\frac12n\right)}\sum_{k=0}^\infty(-1)^k\left(1-\tfrac12n\right)_k\frac{c_k}{\zeta^k}.
\end{equation}

The terms  in  this expansion with index $k\ge m$ vanish when $n=2m$, $m=1,2,3,\dots$ because 
 \begin{equation}\label{eq:delta08}
\frac{1}{2\pi i}\int_{-\infty}^{(0+)} e^{\zeta t}t^{k-m} \,dt=0,\quad k=m, m+1, m+2, \ldots\, .
\end{equation}

With the expansion for $H_n(x;\delta)$ in \eqref{eq:delta07}, we use $F_n(x;\delta)\sim 2H_n(x;\delta)$  for large values of $\delta$ to obtain the expansion in \eqref{eq:delta01}. This follows from the second relation in \eqref{eq:contour16} and from Theorem~\ref{them:them05}, where we have given an asymptotic estimate of  $R_n(x;\delta)$ with exponential factor $e^{-\frac12\delta^2}$. In \eqref{eq:delta01} we see an  exponential factor $e^{-\frac12(1-y)\delta^2}$. Hence, from an asymptotic point of view, when $\delta$ is large  $R_n(x;\delta)$ can be neglected in the relation 
$F_n(x;\delta)=2H_n(x;\delta)-R_n(x;\delta)$.
\eoproof
\end{proof}

In Table~\ref{tab:table01} we give numerical values of $F_n(x;\delta)$ for $n=10.3$, $\delta=20$ and several values of $x$ by using the expansion of $F_n(x;\delta)$ given in \eqref{eq:delta01}.
Because of  \eqref{eq:contour05} we have for large $\delta$ a similar asymptotic estimate $F_n(x;\delta)\sim2P_n(x;\delta)$ as for $H_n(x;\delta)$, and we can compare the computed values with twice the series expansion of $P_n(x;\delta)$ in \eqref{eq:intro01}. The final column in the table gives the number of terms $j_{\max}$ needed in that expansion to have the relative error smaller than $10^{-16}$ compared with the computed sum. We have computed  the incomplete beta functions per term, not with recursion. We see that when $x=\delta$, the transition value, $F_n(x;\delta)$ is near $\frac12$.  As $x$ increases from 5 to the transition value 20,  with decreasing values of  $1-y=n/(n+x^2)$,  the relative accuracy becomes larger. This happens because $c_k$ becomes larger as $1-y$ becomes smaller, which follows from the recursion in \eqref{eq:delta03}. The computations are done with Maple, $Digits=16$.

\renewcommand{\arraystretch}{1.25}
\begin{table}
\caption{Computed values of $F_n(x;\delta)$ and relative errors for  $n=10.3$, $\delta=20$ and several values of $x$.  We used the expansion given in \eqref{eq:delta01} with 11  terms. 
\label{tab:table01}}
$$
\begin{array}{rcccccc}\hline
x \ & F_n(x;\delta) \ & {\rm rel.\  accuracy} \  &1-y& j_{\max} \\
\hline\\[-10pt]
5& 0.7890745035061528\times 10^{-20} & 0.20\times 10^{-13}& 0.29178 & 254\\
8 & 0.1902963697413609\times 10^{-07} & 0.40\times 10^{-12} & 0.13863 & 294\\
11 & 0.4649258368179092\times 10^{-03} & 0.12\times 10^{-09} & 0.07845 & 310\\
14 & 0.2912746016055676\times 10^{-01} & 0.11\times 10^{-07} & 0.04993& 317\\
17 & 0.1858422833307925\times 10^{-00} & 0.41\times 10^{-06} & 0.03441 & 321\\
20 & 0.4434882973203470\times 10^{-00} & 0.82\times 10^{-05} & 0.02510 & 323\\
\hline
\end{array}
$$
\end{table}
\renewcommand{\arraystretch}{1.0}

\subsection{An expansion in terms of incomplete gamma functions}\label{sec:xdelta}

Clearly, the coefficients in the expansion given in Theorem~\ref{them:them06} are not defined as $y\to 1$, which corresponds with $x^2\gg n$. To handle this, we can use the following theorem.
\begin{theorem}\label{them:them07}
The function $F_n(x;\delta)$ has the asymptotic expansions
\begin{equation}\label{eq:delta09}
F_n(x;\delta)\sim y^{\frac12n-\frac12}\sum_{k=0}^\infty \binom{\frac12n-\frac12}{k} \eta^k
Q\left(\tfrac12n-k,\eta\zeta\right),
\end{equation}
and
\begin{equation}\label{eq:delta10}
F_n(x;\delta)\sim1- y^{\frac12n-\frac12}\sum_{k=0}^\infty \binom{\frac12n-\frac12}{k} \eta^k
P\left(\tfrac12n-k,\eta\zeta\right),
\end{equation}
as $ \zeta\to\infty$, where $P(a,z)$ and $Q(a,z)$ are the incomplete gamma function ratios given in \eqref{eq:realint05}, and 
\begin{equation}\label{eq:delta11}
\zeta=\tfrac12\delta^2y,\quad  \eta=\frac{1-y}{y}=\frac{n}{x^2}.
\end{equation}

\end{theorem}

\begin{remark}\label{rem:rem04}
Before we give the proof, we note the following. 
\begin{enumerate}
\item
The expansions in \eqref{eq:delta09} and \eqref{eq:delta10} can be used for $n={\cal{O}}(1)$ and for $\eps \le y \le 1$, where $\eps$ is a fixed positive small number. 
\item
These expansions are finite with $(n-1)/2+1$ terms when $n$ is an odd positive integer. This follows from the definition of the coefficients $a_k$
in \eqref{eq:delta04}. 
\item
The expansion in  \eqref{eq:delta09} is finite with $n/2 $ terms  when $n$ is an even positive integer, because $Q(a,z)=0$ for $a=0,-1,-2,\ldots$, which follows from the definition  in \eqref{eq:realint05}.
\item
Because $Q(a,z)=\Gamma(a,z)/\Gamma(a)\sim z^{a-1}e^{-z}/\Gamma(a)$ as $z\to\infty$, we have the estimate
\begin{equation}\label{eq:delta12}
\eta^k Q\left(\tfrac12n-k,\eta\zeta\right)\sim\frac{(\eta\zeta)^{\frac12n-1} e^{-\eta\zeta}}{\Gamma\left(\frac12n-k\right)} \zeta^{-k},\quad \zeta\to\infty,
\end{equation}
which shows that the expansion in \eqref{eq:delta09} has a true asymptotic character for large values of $\zeta=\frac12\delta^2y$.
\item
The asymptotic character of the expansion in  \eqref{eq:delta10} follows from (see \cite[Eq.~13.6.5, Eq.~13.7.1]{Olde:2010:CHF})
\begin{equation}\label{eq:delta13}
P(a,z)=\frac{z^a e^{-z}}{\Gamma(a+1)}\CHF{1}{a+1}{z}\sim 1,\quad z\to\infty,
\end{equation}
with $a$ fixed. This gives for fixed $n$ and $k$
\begin{equation}\label{eq:delta14}
\eta^k P\left(\tfrac12n-k,\eta\zeta\right)\sim \eta^k=(n/x^2)^k,\quad \eta\zeta\to\infty,\quad  \eta\zeta=\tfrac12\delta^2 
\frac{n}{n+x^2}.
\end{equation}
Now the assumed large values of $x$ are relevant for the asymptotic character of the expansion. When $\eta\zeta=\bigO(1)$, $P(\tfrac12n-k,\eta\zeta)=\bigO(1)$, and the asymptotic character of the expansion again comes from the powers $\eta^k$.
 \end{enumerate}
\end{remark}

\begin{proof} 
Again we start with a result for the function $H_n(x;\delta)$ and use the integral  representation in \eqref{eq:delta05}.
After substituting the expansion of $(1+t)^{\frac12n-\frac12}$ (see  \eqref{eq:delta04}), we find an expansion with integrals of the form
\begin{equation}\label{eq:delta15}
\frac{1}{2\pi i}\int_{-\infty}^{(0+)} 
e^{\zeta t} t^{-\frac12n+k}\,\frac{dt}{\eta-t},
\end{equation}
with $\zeta$ and $\eta$ given in \eqref{eq:delta11}. These integrals can be written in terms of the incomplete gamma function $Q(a,z)$. We use the representation (see  \cite[Section~37.2]{Temme:2015:AMI})
\begin{equation}\label{eq:delta16}
Q(a,z)=\frac{e^{-a\phi(\lambda)}}{2\pi i}\int_{-\infty}^{(0+)} 
e^{a\phi(s)}\,\frac{ds}{\lambda-s},
\end{equation}
where the contour cuts the positive real axis in the interval $[0,\lambda]$  and
\begin{equation}\label{eq:delta17}
\phi(s)=s-1-\ln s,\quad \lambda=\frac{z}{a}.
\end{equation}
Using $\eta\zeta=\frac12(1-y)\delta^2$ and the substitution $t=(\frac12n-k)s/\zeta$  in the integral in \eqref{eq:delta15},  the expansion is as given in \eqref{eq:delta09}, where we have used $F_n(x;\delta)\sim2H_n(x;\delta)$ again.

The  expansion in \eqref{eq:delta10} follows from shifting the contours in  \eqref{eq:contour15} and 
 \eqref{eq:delta16} to the right, across the poles, picking up the residues and using
the complementary relation $Q(a,z)=1-P(a,z)$. A formal proof follows  from this relation and the expansion $\dsp{\sum_{k=0}^\infty a_k \eta^k=y^{-\frac12n+\frac12}}$. 
\eoproof
\end{proof}

Recurrence relations to compute the normalised incomplete gamma functions are given in \cite[\S8.8]{Paris:2010:INC}. 
 Which one of the expansions in Theorem~\eqref{them:them07} is most convenient for numerical computations follows from choosing the series with the smallest first term. We know that $Q\left(\frac12n,\eta\zeta\right) \le P\left(\frac12n,\eta\zeta\right)$ when (roughly) $\frac12n \le \eta\zeta$. That is, when $x^2+n\le \delta^2$. When $x^2\gg n$ and $\delta$ is also large, we conclude that for $x\le \delta$ we should take \eqref{eq:delta09}, otherwise it is better to take expansion \eqref{eq:delta10}. 

When we compare Table~\ref{tab:table01} with Table~\ref{tab:table02} we observe for the same parameters a uniform relative error in the latter. Also,  in  \eqref{eq:delta09} we used only 6 terms. The computations are done with Maple, $Digits=16$, and we evaluated the incomplete gamma functions by Maple's $\Gamma(a,z)$; no recursion is used. For comparison, we have used  twice the expansions of $P_n(x;\delta)$ in \eqref{eq:intro01}.

\renewcommand{\arraystretch}{1.25}
\begin{table}
\caption{Computed values of $F_n(x;\delta)$ and relative errors for  $n=10.3$, $\delta=20$ and several values of $x$.  We used the expansion given in \eqref{eq:delta09} with 6 terms. 
\label{tab:table02}}
$$
\begin{array}{rcccccc}\hline
x \ & F_n(x;\delta) \ & {\rm rel.\  accuracy} \  &1-y& j_{\max} \\
\hline\\[-10pt]
5 & 0.7890745035061292 \times 10^{-20} & 0.10 \times 10^{-13} & 0.29178  & 254\\
8 & 0.1902963697414361 \times 10^{-07} & 0.32 \times 10^{-14} & 0.1386 & 294\\
11 & 0.4649258368729340 \times 10^{-03} & 0.97 \times 10^{-14} & 0.07845 & 310\\
14 & 0.2912746047441838 \times 10^{-01} & 0.69 \times 10^{-14} & 0.04993 & 317\\
17 & 0.1858423597361172 \times 10^{-00} & 0.70 \times 10^{-14} & 0.03441 & 321\\
20 & 0.4434919419501216  \times 10^{-00} & 0.50 \times 10^{-14} & 0.02510  & 323\\
\hline
\end{array}
$$
\end{table}
\renewcommand{\arraystretch}{1.0}

\section{Asymptotic expansions for large \protectbold{n}}\label{sec:nlarge}

For large values of $n$ we consider two cases: one with $\delta$ and $x$ small with respect to $n$, and a second case with larger values of these parameters.

\subsection{Bounded values of \protectbold{\delta} and \protectbold{x}}\label{sec:bounded}
We use the representation of $F_n(x;\delta)$ in \eqref{eq:realint03} and write
\begin{equation}\label{eq:largen01}
\begin{array}{@{}r@{\;}c@{\;}l@{}}
F_n(x;\delta)&=&\dsp{A_ne^{-\frac12n}\int_0^\infty E(t)
e^{-\frac12n\phi(t)}\,\frac{dt}{t},}  \\[8pt]
A_n&=&\dsp{\frac{(n/2)^{n/2} }{\Gamma(n/2)},\quad E(t)=\erfc\left((\delta-xt)/\sqrt{2}\right),}\\[8pt]
\phi(t)&=&\dsp{t^2-\ln \left(t^2\right)-1,\quad \phi^{\prime}(t)=\frac{2(t^2-1)}{t}.}
\end{array}
\end{equation}
The asymptotic analysis focuses on the positive saddle point $t=1$ and to represent the coefficients in a convenient form we first split off a term by writing $E(t)=\left(E(t)-E(1)\right)+E(1)$, where $E(1)=\erfc\left((\delta-x)/\sqrt{2}\right)$. The part with $E(1)$ gives
\begin{equation}\label{eq:largen02}
E(1) A_n\int_0^\infty 
e^{-\frac12n t^2} t^{n-1}\,dt=\tfrac12\erfc\left((\delta-x)/\sqrt{2}\right),
\end{equation}
and we continue with the following theorem.
\begin{theorem}\label{them:them08}
For large values of $n$ and bounded values of $\delta$ and $x$ we have the expansion
\begin{equation}\label{eq:largen03}
F_n(x;\delta)\sim \tfrac12\erfc\left((\delta-x)/\sqrt{2}\right)+B(x;\delta)\sum_{k=1}^\infty\frac{c_k}{n^k},\quad n\to\infty,
\end{equation}
where $B(x;\delta)$ and the first coefficients $c_k$ are given by
\begin{equation}\label{eq:largen04}
\begin{array}{ll}
&B(x;\delta)=\dsp{x\sqrt{\frac{2}{\pi}}\,e^{-\frac12(\delta-x)^2}},\\[8pt]
&c_1=\frac18\left(x\delta-x^2-1\right),\\[8pt]
&c_2=\frac{1}{192}\left(3\delta^3 x^3 - 9\delta^2 x^4 + 9\delta x^5 - 3x^6 - 2\delta^2 x^2 - 5\delta x^3 + 7x^4 -3 \delta x +5 x^2 +3\right).
\end{array}
\end{equation}
\end{theorem}

\begin{proof}
With the transformation 
\begin{equation}\label{eq:largen05}
\phi(t)=\tfrac12 w^2,\quad \sign(t-1)=\sign(w),
\end{equation}
we obtain
\begin{equation}\label{eq:largen06}
\begin{array}{@{}r@{\;}c@{\;}l@{}}
F_n(x;\delta)&=&\dsp{ \tfrac12\erfc\left((\delta-x)/\sqrt{2}\right)+A_ne^{-\frac12n}\int_{-\infty}^\infty e^{-\frac14nw^2} f(w)\,dw,}\\[8pt]
f(w)&=&\dsp{\frac{E(t)-E(1)}{t}\frac{dt}{dw}.}
\end{array}
\end{equation}
We invert the relation in \eqref{eq:largen05} in the form of the expansion 
\begin{equation}\label{eq:largen07}
t=1+\tfrac12w+\tfrac{1}{24} w^2+\bigO(w^3),
\end{equation} 
which is valid for small $\vert w\vert$. Next we expand for small $\vert t-1\vert$
\begin{equation}\label{eq:largen08}
\frac{E(t)-E(1)}{t}=B(x;\delta)(t-1)\left(1+\tfrac12(\delta x-x^2-2)(t-1)+\bigO(t-1)^2)\right),
\end{equation}
where $B(x;\delta)$ is given in \eqref{eq:largen04}. 
With these two expansions we obtain
\begin{equation}\label{eq:largen09}
f(w)=\sum_{k=0}^\infty f_k w^k=B(x;\delta)\left(\tfrac14w+\tfrac{1}{16}(\delta x-x^2-1)w^2+\bigO(w^3)\right),
\end{equation}
from which we can obtain the first $f_k$. With the expansion of $f(w)$ we find the asymptotic expansion
\begin{equation}\label{eq:largen10}
F_n(x;\delta)\sim\tfrac12\erfc\left((\delta-x)/\sqrt{2}\right)+ \frac{B(x;\delta)}{\Gamma^*(n/2)}\sum_{k=0}^\infty\frac{b_k}{n^k},
\end{equation}
where
\begin{equation}\label{eq:largen11}
b_0=0,\quad b_k=\frac{1}{B(x;\delta)}\left(\tfrac12\right)_k2^{2k}f_{2k},\quad k=1,2,3,\ldots\,.
\end{equation}
Here we used (see \eqref{eq:largen01})
\begin{equation}\label{eq:largen12}
2A_ne^{-\frac12n}\sqrt{\frac{\pi}{n}}=2\frac{(n/2)^{n/2} }{\Gamma(n/2)}e^{-\frac12n}\sqrt{\frac{\pi}{n}}=\frac{1}{\Gamma^*(n/2)},
\end{equation}
and the scaled gamma function $\Gamma^*(z)$ defined by
\begin{equation}\label{eq:largen13}
\Gamma(z)=\sqrt{2\pi}\,e^{-z}z^{z-\frac12}\Gamma^*(z).
\end{equation}
 The first coefficients $b_k$ are 
\begin{equation}\label{eq:largen14}
\begin{array}{@{}r@{\;}c@{\;}l@{}}
b_0&=&0,\quad \dsp{b_1= \tfrac18\left(x\delta-x^2-1\right),}\\[8pt]
b_2&=& \frac{1}{192}\left(3\delta^3 x^3 - 9\delta^2 x^4 + 9\delta x^5   - 3x^6 - 2\delta^2 x^2-5\delta x^3  + 7x^4 + \delta x + x^2 - 1\right).
\end{array}
\end{equation}
Because $n$ is large we can use the asymptotic expansion of the scaled gamma function. We have
(see \cite[Chapter~6]{Temme:2015:AMI})
\begin{equation}\label{eq:largen15}
\Gamma^*\left(\tfrac12n\right)\sim\sum_{k=0}^\infty \frac{a_k}{n^k},\quad \frac{1}{\Gamma^*\left(\frac12n\right)}\sim\sum_{k=0}^\infty (-1)^k\frac{a_k}{n^k},\quad n\to\infty,
\end{equation}
with first coefficients $a_0 = 1$, $a_1 = \frac{1}{6}$, $a_2 = \frac{1}{72}$, $a_3=\tfrac{139}{6480}$,
and we can write the expansion in \eqref{eq:largen10} as in \eqref{eq:largen03},
where the $c_k$ follow from the formal power series multiplication 
\begin{equation}\label{eq:largen16}
\sum_{k=1}^\infty \frac{c_k}{n^k}=\sum_{k=0}^\infty (-1)^k \frac{a_k}{n^k}\sum_{k=1}^\infty \frac{ b_k}{n^k},
\end{equation}
which gives
\begin{equation}\label{eq:largen17}
c_0=0,\quad c_k=\sum_{j=0}^{k-1} (-1)^ja_j b_{k-j}, \quad k\ge1.
\end{equation}
\eoproof
\end{proof}

 The coefficients $f_k$ and  $b_k$ appearing in this proof can be obtained by using manipulation of power series. Sometimes, simple explicit forms or recurrence relations for these coefficients are available, but not in the present case. It is always possible to construct explicit formulas for these coefficients. For a short summary of this topic we refer to \cite[Section~3.3]{Temme:2015:AMI}.

\renewcommand{\arraystretch}{1.25}
\begin{table}
\caption{Computed values of $F_n(x;\delta)$ and relative errors for  $n=1000$, $\delta=10$ and several values of $x$.  We used  the expansion given in  \eqref{eq:largen03} with 6 terms. 
\label{tab:table03}}
$$
\begin{array}{rccrcc}\hline
x \ & F_n(x;\delta) \ & {\rm rel.\  accuracy} \  &\ \  j_{\max} \\
\hline\\[-10pt]
1.0 & 0.1149355213382657\times 10^{-18} & 0.78\times 10^{-14} & 23\\
2.5 & 0.3472666641712521\times 10^{-13} & 0.31\times 10^{-11} & 39\\
5.0 & 0.3344736497826102\times 10^{-06} & 0.15\times 10^{-10} & 60\\
7.5 & 0.6806053041739461\times 10^{-02} & 0.17\times 10^{-11} & 80\\
10.0 & 0.4990113438101769\times 10^{-00}& 0.21\times 10^{-10} & 97\\
12.5 & 0.9919154834852601\times 10^{-00} & 0.63\times 10^{-10} & 110\\
\hline
\end{array}
$$
\end{table}
\renewcommand{\arraystretch}{1.0}

In Table~\ref{tab:table03} we give numerical values of $F_n(x;\delta)$ for $n=1000$, $\delta=10$ and several values of $x$. We have obtained  $F_n(x;\delta)$ using  the expansion given in  \eqref{eq:largen03} with 6 terms. We see that when $x=\delta$, the transition value, $F_n(x;\delta)$ is near~$\frac12$. 

We  have compared the computed values with the representation of $F_n(x;\delta)$ in \eqref{eq:intro01}. The final column in the table gives the number of terms $j_{\max}$ needed in the expansions to have the relative error smaller than $10^{-16}$ compared with the computed sum. We have computed  the incomplete beta functions per term, not with recursion. These computations are done with Maple, $Digits=16$.
The computed value for $x=1$ can be compared with  the result 
given in \cite[Table~1]{Witkovsky:2013:NOC}  for the same chosen parameters.

\subsection{Values of  \protectbold{\delta} and \protectbold{x} of order \protectbold{\bigO(n)}}\label{sec:unrestrict}

To obtain the result for  $F_n(x;\delta)$ we start with $P_n(x;\delta)$ and use the relation $F_n(x;\delta)=2P_n(x;\delta)+R_n(x;\delta)$ derived in Lemma~\ref{lem:lem03}. We scale the $x$ and $\delta$ parameters by writing
\begin{equation}\label{eq:largen18}
x=\xi\sqrt{n},\quad \delta=\sigma\sqrt{n},
\end{equation}
and we assume that $x/n$ and $\delta/n$ are bounded. For small values of these parameters we refer to Section \ref{sec:bounded}. We have the following theorem. 

\begin{theorem}\label{them:them09}
Consider  the integral of $P_n(x;\delta)$ given in \eqref{eq:contour13} written in the form
\begin{equation}\label{eq:largen19}
\begin{array}{@{}r@{\;}c@{\;}l@{}}
P_n(x;\delta)&=&\dsp{\tfrac12\frac{\sqrt{y}\,e^{-\frac12n\phi(1/y)}}{2\pi i}
\int_0^{(1+)}e^{\frac12n\phi(t)}\,\frac{dt}{(1-yt)\sqrt{t}},}\\[8pt]
\phi(t)&=&\dsp{z t+\ln(t)-\ln(t-1),\quad y=\frac{x^2}{n+x^2}=\frac{\xi^2}{1+\xi^2},\quad \quad z= y\sigma^2.}
\end{array}
\end{equation}
Let $t_0$ be the   positive saddle point  that follows from 
\begin{equation}\label{eq:largen20}
\phi^{\prime}(t)=\frac{z t^2-z t-1}{t(t-1)} \quad \Longrightarrow \quad   t_0=\frac{z+\sqrt{z(z+4)}}{2z},
\end{equation}
and let the real number $\eta$ be defined by
\begin{equation}\label{eq:largen21}
\phi(1/y)-\phi(t_0)=\eta^2, \quad \sign(\eta)=\sign(1/y-t_0).
\end{equation}
Then, as $n\to\infty$ and $z>0$ we have the asymptotic result
\begin{equation}\label{eq:largen22}
P_n(x;\delta)\sim\tfrac14\erfc\left(\eta\sqrt{n/2}\right)+
\tfrac12\frac{e^{-\frac12n\eta^2}}{\sqrt{2\pi n}} \sum_{k=0}^\infty \frac{c_k}{n^k},  \quad n\to\infty.
\end{equation}
The first two coefficients $c_k$ are given in the proof.
\end{theorem}

Before we give the proof we note that the location of the saddle point $t_0$ relative to the pole at $1/y$ has a major influence on the asymptotic behaviour of $P_n(x;\delta)$. We can verify that  $t_0=1/y$ if $\sigma=\xi$, that is, if $\delta=x$, which is the transition case; see the end of Section~\ref{sec:realint}. Initially we assume that $t_0<1/y$, in which case we can take the contour in 
\eqref{eq:largen19} through the saddle point $t_0$, which is always larger than unity. We have  $\sign(t_0-1/y)=\sign(x-\delta)$, as is easily verified.

\begin{proof}
We use the transformation 
\begin{equation}\label{eq:largen23}
\phi(t)-\phi(t_0)= s^2, 
\end{equation}
with condition $\sign(s)=\sign(t-t_0)$ for $t>1$ and real $s$, with continuity for complex values, and we obtain
 \begin{equation}\label{eq:largen24}
\begin{array}{@{}r@{\;}c@{\;}l@{}}
P_n(x;\delta)&=&\dsp{\tfrac12\frac{e^{-\frac12n\left(\phi(1/y)-\phi(t_0)\right)}}{2\pi i}\int_{-i\infty}^{+i\infty}e^{\tfrac12n s^2} f(s)\,ds,}\\[8pt]
f(s)&=&\dsp{ \frac{\sqrt{y}}{(1-yt)\sqrt{t}}\frac{dt}{ds}}=\dsp{ \sqrt{\frac{\rho}{t}}\,\frac{1}{\rho-t}\frac{dt}{ds},\quad \rho=\frac{1}{y}.}
\end{array}
\end{equation}
The pole at $1/y$ in the $t$-plane should correspond with a pole, say $\eta$, in the $s$-plane, and $\eta$ follows from the equation given in \eqref{eq:largen21}, where the sign condition follows from the one given for the transformation in \eqref{eq:largen21}.
Because we have assumed $\delta > x$, hence $1/y>t_0$, it follows that $\eta>0$ (for the time being).

We split off the pole, writing
\begin{equation}\label{eq:largen25}
f(s)=\frac{A}{\eta-s}+g(s),
\end{equation}
assuming that $g(s)$ is finite at $s=\eta$. Writing 
\begin{equation}\label{eq:largen26}
A=\frac{\eta-s}{\rho-t}\sqrt{\frac{\rho}{t}}\frac{dt}{ds}-(\eta-s)g(s),
\end{equation}
and using l'H{\^o}pital's rule in the limit $\eta\to s$, $\rho\to t$, we find $A=1$. This gives
\begin{equation}\label{eq:largen27}
P_n(x;\delta)=\tfrac14\erfc\left(\eta\sqrt{n/2}\right)+
\tfrac12\frac{e^{-\frac12n\eta^2}}{2\pi i} 
\int_{-i\infty}^{+i\infty}e^{\frac12n s^2} g(s)\,ds,
\end{equation}
where we have used (see  \cite[7.2.3, 7.7.2]{Temme:2010:ERF})
\begin{equation}\label{eq:largen28}
\frac{e^{-\frac12n\eta^2}}{4\pi i}
\int_{-i\infty}^{+i\infty}e^{\frac12n s^2} \frac{ds}{\eta-s}=\frac{e^{-\frac12n\eta^2}}{4\pi i}
\int_{-\infty}^{\infty}e^{-\frac12n t^2} \frac{dt}{t-i\eta}=\tfrac14\erfc\left(\eta\sqrt{n/2}\right).
\end{equation}
The expansion  ${\dsp g(s)=\sum_{k=0}^\infty g_k s^k} $ gives the asymptotic expansion in \eqref{eq:largen22},
where
\begin{equation}\label{eq:largen29}
c_k=(-1)^k2^k\left(\tfrac12\right)_k g_{2k},\quad k=0,1,2,\ldots\,.
\end{equation}

The coefficients $c_k$  follow from  inverting the relation in \eqref{eq:largen23} by determining the coefficients $t_k$ in the expansion $\dsp{t=t_0 +\sum_{k=0}^\infty t_k s^k}$. The first few are
\begin{equation}\label{eq:largen30}
\begin{array}{@{}r@{\;}c@{\;}l@{}}
t_1&=&\dsp{\frac{\sqrt{2}}{\sqrt{\phi^{\prime\prime}(t_0)}}=\frac{\sqrt{2}}{\sqrt{z\sqrt{z(z+4)}}}=
\frac{\sqrt{2}\,t_0(t_0-1)}{\sqrt{2t_0-1}},}\\[8pt] 
t_2&=&\dsp{\frac{ t_1^2(3t_0^2 - 3t_0 + 1)}{3t_0(t_0 - 1)(2t_0 - 1)}, }\quad
\dsp{t_3=\frac{t_1^3(18t_0^4 - 36t_0^3 + 24t_0^2 - 6t_0 + 1)}{36t_0^2(2t_0 - 1)^2(t_0 - 1)^2},}\\[8pt]
t_4&=&\dsp{-\frac{t_1^4(9t_0^2-9t_0+1)}{270t_0^3(2t_0-1)^3(t_0-1)^3},}
\end{array}
\end{equation}
with $t_0$ defined in \eqref{eq:largen20}. For the coefficients $t_k$ with $k\ge2$ we prefer to express $t_k$  in terms of $t_0$, not $z$, to avoid square roots.
For the first $g_{2k}$ we find
\begin{equation}\label{eq:largen31}
\begin{array}{@{}r@{\;}c@{\;}l@{}}
g_0&=&\dsp{\sqrt{\frac{\rho}{t_0}}\,\frac{t_1}{\rho-t_0}-\frac{1}{\eta},}\\[8pt]
g_2&=&\dsp{\sqrt{\frac{\rho}{t_0}}\,\frac{t_1^3\left(a_0 + a_1t_0 +a_2t_0^2 + a_3t_0^3 +a_4t_0^4 +a_5t_0^5\right)}{24t_0^{2}(t_0 - 1)^2(2t_0 - 1)^2(\rho-t_0)^3}-\frac{1}{\eta^3},}\\[8pt]
a_0&=&-\rho^2, \quad a_1=2\rho(7+3\rho),  \quad  a_2= -3\rho^2 - 84\rho + 11,\\[8pt] 
a_3&=&186\rho - 66, \quad a_4= -216\rho + 129, \quad a_5= 96\rho - 72.
\end{array}
\end{equation}

To obtain these we have used  \eqref{eq:largen24},  \eqref{eq:largen25} and  \eqref{eq:largen30}.
 The coefficients $g_{2k}$ are bounded as $\eta\to0$, and for small values of $\eta$ we can expand the coefficients in powers of $\eta$. Details are given in the Appendix.
\eoproof
\end{proof}

\begin{remark}\label{rem:rem05}
We have assumed $\eta>0$ in the proof, which follows from the initial assumption $\delta>x$. The function $g(s)$ introduced in \eqref{eq:largen25} is analytic at $s=\eta$ (whether or not positive or zero) and the complementary error function in \eqref{eq:largen22} is also and analytic function of $\eta$.  It follows that we can use the asymptotic result in  \eqref{eq:largen22} also for $\delta\le x$ ($\eta\le0$), although for $\delta < x$ it is better to use the complementary function   $G_n(x;\delta)=1-F_n(x;\delta)=1-2P_n(x;\delta)-R_n(x;\delta)$, and the expansion
\begin{equation}\label{eq:largen32}
1-2P_n(x;\delta)\sim\tfrac12\erfc\left(-\eta\sqrt{n/2}\right)-
\frac{e^{-\frac12n\eta^2}}{\sqrt{2\pi n}} \sum_{k=0}^\infty \frac{c_k}{n^k},  \quad n\to\infty.
\end{equation}
\end{remark}

\begin{example}\label{ex:ex01}
When we take $x=1000$, $n=1000$, $\delta=1010$ and we compute $\eta\doteq 0.014104716$ from \eqref{eq:largen21}.  With
\begin{equation}\label{eq:largen33}
F_n(x;\delta)\sim 2P_n(x;\delta)\sim\tfrac12\erfc\left(\eta\sqrt{n/2}\right)+
\frac{e^{-\frac12n\eta^2}}{\sqrt{2\pi n}} c_0,
\end{equation}
the result  is  $F_n(x;\delta)\doteq0.3224383497$, with relative error $1.96\times 10^{-7}$. With $x=500$, $n=100$, $\delta=510$ we find $\eta\doteq0.028180106$ and $F_n(x;\delta)\doteq0.3711630202$, with relative error $5.51\times 10^{-6}$. We used Maple, $Digits=16$, and compared with values given in  \cite[Table~1]{Witkovsky:2013:NOC}.
\end{example}

\renewcommand{\arraystretch}{1.25}
\begin{table}
\caption{Relative errors of numerical results  for   several values of $x$, $n$ and $\delta$.  We used twice the expansion given in \eqref{eq:largen22} with 3 terms and 6 terms. 
\label{tab:table04}}
$$
\begin{array}{rrrcccc}\hline
x\   & n \ & \delta\  & 3\ {\rm terms} &6\  {\rm terms} \\
\hline\\[-10pt]
50& 100 & 75& 1.26\times 10^{-08}& 6.38\times10^{-09}\\
500 &100& 510& 7.45\times 10^{-10}& 4.54\times10^{-12}\\
100 & 1000 & 105& 6.94\times 10^{-13}& 3.00\times10^{-15}\\
1000 &1000 & 1010&2.65\times 10^{-13} & 3.00\times10^{-15}\\
\hline
\end{array}
$$
\end{table}
\renewcommand{\arraystretch}{1.0}

In Theorem~\ref{them:them05} we have given an estimate of $R_n(x;\delta)$ for large values of $\delta$. This estimate can be used to decide if we need the computation of the function $R_n(x;\delta)$  in the algorithm to compute $F_n(x;\delta) =2P_n(x;\delta) + R_n(x;\delta)$. 

In Table~\ref{tab:table04} we give relative errors of numerical results of the expansion in  \eqref{eq:largen22} with 3 and with 6 terms. We used twice the results to obtain an approximation of $F_n(x;\delta)$. In these examples  we have not used the term $R_n(x;\delta)$ because it is too small to influence the obtained accuracy. We used Maple, $Digits=16$, and compared with values given in  \cite[Table~1]{Witkovsky:2013:NOC}. For the values of $x$ and $\delta$ of the table the values of $\eta$ defined in \eqref{eq:largen21} are rather small, and we have used expansions ${g_{2k}=\sum_{j=0}^{\infty} g_{j,k}\eta^k}$, as explained in the Appendix.

To complement the expansion in \eqref{eq:largen22} we conclude this section with deriving an asymptotic expansion of $R_n(x;\delta)$ valid for large values of $n$.  Again we use the scaled parameters
$\xi=x/\sqrt{n}$ and $\sigma=\delta/\sqrt{n}$. 

\begin{theorem}\label{them:them10}
Let us define
\begin{equation}\label{eq:largen34}
\psi(t)=zt-\ln t+\ln(1+t),\quad t>0, \quad z=y\sigma^2.
\end{equation}
Then, as $n\to\infty$ and  $z > 0$, the function  $R_n(x;\delta)$ with integral representation given in  \eqref{eq:contour11}  has the asymptotic expansion
\begin{equation}\label{eq:largen35}
R_n(x;\delta)\sim\frac{e^{-\frac12\delta^2}\sqrt{y}\,(1-y)^{\frac12n}}{\sqrt{2\pi n}}
e^{-\frac12n\psi(t_p)}\sum_{k=0} ^\infty \frac{d_k}{n^k},\quad
\end{equation}
where $t_p$ is the positive saddle point that follows from
\begin{equation}\label{eq:largen36}
\psi^\prime(t)=\frac{zt^2+zt-1}{t(1+t)} \quad \Longrightarrow \quad t_p=\frac{-z+\sqrt{z(z+4)}}{2z}.
\end{equation}
The first coefficients $d_k$ are given in the proof.
\end{theorem}

We assume that $x$ and $\delta$ are bounded. In particular, when $\delta$ is large, $z$ will be large, and  the saddle point $t_p$ will approach the origin, where the integral has a singularity. Other asymptotic methods can be used in which $K$-Bessel functions are the main approximants, similarly as for the Kummer $U$ function, as discussed in \cite[Section~27.4.1]{Temme:2015:AMI}.

\begin{proof}
We write the integral of $R_n(x;\delta)$  as
\begin{equation}\label{eq:largen37}
R_n(x;\delta)=\frac{e^{-\frac12\delta^2}\sqrt{y}\,(1-y)^{\frac12n}}{2\pi}\int_0^\infty e^{-\frac12n\psi(t)}
\,\frac{dt}{\sqrt{t}\,(1+yt)}.
\end{equation}
The transformation 
\begin{equation}\label{eq:largen38}
\psi(t)-\psi(t_p)=w^2,\quad \sign(t-t_p)=\sign(w),
\end{equation}
gives
\begin{equation}\label{eq:largen39}
\begin{array}{@{}r@{\;}c@{\;}l@{}}
R_n(x;\delta)&=&\dsp{\frac{e^{-\frac12\delta^2}\sqrt{y}\,(1-y)^{\frac12n}}{2\pi}e^{-\frac12n\psi(t_p)}\int_{-\infty}^\infty e^{-\frac12nw^2} h(w)\,dw,}\\[8pt]
 h(w)&=&\dsp{ \frac{1}{\sqrt{t}\,(1+yt)}\frac{dt}{dw}.}
\end{array}
\end{equation}
To find the coefficients we can use a standard method, first by determining the coefficients $r_k$ of the expansion $\dsp{t=t_p+\sum_{k=1}^\infty r_k w^k}$ and from these we can find the coefficients $h_k$ of the  expansion $\dsp{h(w)=\sum_{k=0}^\infty h_k w^k}$. Substitution of this series in \eqref{eq:largen39}  gives the asymptotic expansion in \eqref{eq:largen35} with coefficients $\dsp{ d_k=2^k\left(\tfrac{1}{2}\right)_kh_{2k}}$.  The first coefficients are
\begin{equation}\label{eq:largen40}
\begin{array}{@{}r@{\;}c@{\;}l@{}}
d_0&=&h_0=\dsp{\frac{r_1}{\sqrt{t_p}\,(1+yt_p)},}\\[8pt]
r_1&=&\dsp{\frac{\sqrt{2}}{\sqrt{\psi^{\prime\prime}(t_p)}}=\frac{\sqrt{2}\,t_p(t_p+1)}{\sqrt{2t_p+1}}=\frac{\sqrt{2}}{\sqrt{z\sqrt{z(z+4)}}},}\\[8pt]
r_2&=&=\dsp{\frac{r_1^2(3t_p^2+3t_p+1)}{3t_p(t_p+1)(2t_p+1)}},\\[8pt]
d_1&=&\dsp{h_2=r_1^3\frac{b_0+b_1t_p+b_2t_p^2+b_3t_p^3+b_4t_p^4+b_5t_p^5}{24t_p^{5/2}(t_p +1)^2(2t_p + 1)^2(t_py + 1)^3},}\\[8pt]
b_0&=&-1, \quad b_1=-14y-6, \quad  b_2=11y^2 -84y - 3,\\[8pt] 
b_3&=&6y(11y - 31), \quad b_4=3y(43y - 72), \quad b_5=24y(3y - 4).
\end{array}
\end{equation}
\eoproof
\end{proof}

We expect that the estimate in \eqref{eq:largen33} without the term with $c_0$ can serve as a valuable estimate for a large interval of the parameters $x$, $n$ and $\delta$. It can also be used to obtain a first guess when solving $x$ or $\delta$ of the equation $F_n(x;\delta)=f$, with $f\in(0.1) $, because inversion of the complementary error function is a very simple problem. Similar asymptotic inversion methods for several classical cumulative distribution functions are considered in \cite[Chapter~42]{Temme:2015:AMI}. For inverting  the non-central beta distribution and the Student's $t$-distribution we refer to \cite{Gil:2019:NCB} and \cite{Gil:2022:NST}. In \cite[Section~16]{Owen:1968:PNT} several approximations of percentage points in terms of elementary expressions are  discussed in with comparisons from the literature factors in tables.

\section{Numerical aspects}\label{sec:num}
We consider a few possible approaches for the numerical evaluation of the  functions $F_n(x;\delta)$, $P_n(x;\delta)$ and $Q_n(x;\delta)$. The asymptotic expansions derived in the previous sections may be very efficient for certain choices of the parameters $x$, $n$ and $\delta$, but we need to use different methods for a substantial remaining part of values of these quantities.

\subsection{Using the defining series}\label{sec:series}
The series expansions in \eqref{eq:intro01} is an excellent starting point to compute the functions 
$P_n(x;\delta)$ and $Q_n(x;\delta)$ when the noncentrality parameter $\delta$ is small. The incomplete beta functions can be computed by using the recurrence relation
\begin{equation}\label{eq:numer01}
(p+j)f_{j+1}=\left(p+j+(p+j+q-1)y\right)f_j-(p+j+q-1)yf_{j-1},
\end{equation}
where $f_j=I_y (p+j,q)$. Initial values for the recursion can be computed using, for example, our recent algorithm \cite{Egorova:2022:CCB}. 
It should be observed that the recursion is not stable in the forward direction, but we can use a backward recursion scheme (see \cite[\S4.6]{Gil:2007:NSF}). Because $0 < I_y (p,q)<1$ the convergence of the series is always better than that of the Taylor series of $e^{\frac12\delta^2}$, and we may obtain a starting value of the backward recursion by examining the series of this exponential function. 

For example, when we take $x= 5$,   $n= 10$, and  $\delta= 7.5$, we need 73 terms  in the series of  $P_n(x;\delta)$ in \eqref{eq:intro01} to get 
\begin{equation}\label{eq:numer02}
\frac{p_{72}I_y(72+\frac12,\frac12n)} {\sum_{j=0}^{72} p_j I_y(j+\frac12,\frac12n)}< \eps, \quad p_j=\frac{z^j}{j!},\quad z=\tfrac12\delta^2,\quad \eps=10^{-16}.
\end{equation}
The smallest $j_0$ for which $e^{-z}\sum_j^{j_0} p_j <\eps$ is $j_0=85$. A simple inversion method to find $j_0$ numerically follows from solving for $j>z$ the equation
\begin{equation}\label{eq:numer03}
e^{-z}\frac{z^j}{e^{-j}j^j}=\eps.
\end{equation}
In a second example we take  $x= 15$,   $n= 510$, and  $\delta= 17.5$, and we need 251 terms  in the series of  $P_n(x;\delta)$ in \eqref{eq:intro01}. By solving the equation in \eqref{eq:numer03} we find $j\doteq 271$.

In this way, by solving  the equation in  \eqref{eq:numer03}, we can obtain an estimate of the $j$-value to start the backward recursion for the evaluation of the series in \eqref{eq:intro01} for both functions  $P_n(x;\delta)$ and $Q_n(x;\delta)$.

In \cite{Posten:1994:ANA} algorithms are given for the numerical evaluation of the series of $P_n(x;\delta)$  in \eqref{eq:intro01} and of the series in \eqref{eq:intro05}.  In an algorithm in \cite{Lenth:1989:CDF} both defining series for $P_n(x;\delta)$ and $Q_n(x;\delta)$ are used.

\subsection{Numerical quadrature}\label{sec:quad}

In \cite{Witkovsky:2013:NOC} numerical quadrature is used based on Gauss-Kronrod numerical integration for the integral given in \eqref{eq:realint04}. The results in Table~1 of that paper show that the method  can be used for a wide range of the parameters. Our integral in \eqref{eq:realint03} is simpler because the incomplete gamma function is replaced by an error function.

In Figure~\ref{fig:fig02} we see quite simple graphs of $F_n(x;\delta)$, which indeed might suggest that quadrature is a simple method when an algorithm for the complementary error function is available. For small and not too large values of the parameters this is true, but for large values the left side of the graph of the bell-shaped curve is very steep due to the influence of the complementary error function from very small values when $t< \delta/x$ to values near 2 for larger values of $t$. The influence of the part $e^{-\frac12nt^2} t^n$, with maximum at $t=1$, is also quite noticeable. 

The trapezoidal rule is a simple quadrature method, and we have verified its performance for several cases. For example, the integral in  \eqref{eq:realint03} becomes after the substitution $t=e^s$
\begin{equation}\label{eq:numer04}
F_n(x;\delta)=A_n \int_{-\infty}^\infty
\erfc\left((\delta-xe^s)/\sqrt{2}\right) e^{-\frac12ne^{2s}} e^{ns}\,ds.
\end{equation}
 When we take $x=1$, $n=10$ and $\delta=5$, use Maple with $Digits=16$, we need 72 function evaluations on the $s$-interval $[-3.975,1.35]$ with step size $h=0.075$ to obtain $F_n(x;\delta)\doteq 0.00004347252856505909$ with relative error $2.0\times 10^{-15}$. 
 We  have compared this value with the corresponding value in Table~1 of  \cite{Witkovsky:2013:NOC}.

\subsection{Asymptotic expansions}\label{sec:numas}
For testing the asymptotic expansions we have provided tables in the earlier sections.

\section{Appendix}\label{sect:append}
We give details about obtaining the coefficients $g_{2k}$ used via \eqref{eq:largen29} in the asymptotic  expansion in \eqref{eq:largen22}, with $g_0$ and $g_2$ given in \eqref{eq:largen31}. In particular we take care for small values of $\eta$, which means $\delta\sim x$, the transition case.

 It follows from the first coefficients and from evaluated later ones that we can  write
\begin{equation}\label{eq:app01}
g_{2k}=\frac{1}{\eta^{2k+1}}\left(\sqrt{\frac{\rho}{t_0}}\frac{G_{2k}(t_0)}{\tau^{2k+1}}-1\right),
\quad \tau=\frac{\rho-t_0}{\eta t_1},
\end{equation}
where $G_{2k}(t_0)$ does not depend on $\eta$. The first $G$-coefficients are (see  \eqref{eq:largen31})
\begin{equation}\label{eq:app02}
G_0(t_0)=1, \quad G_2(t_0)=\frac{a_0+a_1t_0+a_2t_0^2+a_3t_0^3+a_4t_0^4+a_5t_0^5}{24t_0^2(t_0 - 1)^2(2t_0 - 1)^2}.
\end{equation}

Although this representation is not stable for small values of $\eta$, it is somewhat better to handle than the forms in \eqref{eq:largen31}, because the subtraction is now with terms of size unity. We concentrate now on obtaining expansions in powers of $\eta$. 

The coefficients $g_{2k}$ are analytic at $\eta=0$ and for the numerical evaluations we can use expansions 
\begin{equation}\label{eq:app03}
g_{2k}=\sum_{j=0}^{\infty} g_{j,k}\eta^j,\quad k=0,1,2,\ldots\,.
\end{equation}
To obtain the coefficients $g_{j,k}$ we use the relation in  \eqref{eq:largen21}, applied with $1/y=\rho$, from which we obtain the expansions
\begin{equation}\label{eq:app04}
\frac{\rho}{t_0}=1+\sum_{k=1}^\infty \frac{t_k}{t_0}\eta^k,\quad \tau=1+\sum_{k=1}^\infty \frac{t_{k+1}}{t_1}\eta^k,
\end{equation}
with first coefficients $t_k$ given in  \eqref{eq:largen30}. From the  expansions we can easily generate the expansions
\begin{equation}\label{eq:app05}
\sqrt{\frac{\rho}{t_0}}\frac{1}{\tau^{2k+1}}=\sum_{j=0}^\infty p_{j,k}\eta^j, \quad p_{0,k}=1,\quad k=0,1,2,\ldots\,.
\end{equation}
The first terms of the expansion of $G_2$ are
\begin{equation}\label{eq:app06}
G_2=1+\frac{t_1(4t_0^2 - 3t_0 + 1)}{2t_0(2t_0 - 1)(t_0 - 1)}\eta+\bigO\left(\eta^2\right),
\end{equation}
and for $g_2$ we obtain
\begin{equation}\label{eq:app07}
\begin{array}{@{}r@{\;}c@{\;}l@{}}
g_2&=&\dsp{\frac{1}{\eta^3}\left( 1 - \frac{t_1^3\left(540t_0^4 - 1485t_0^3 + 1269t_0^2 - 369t_0 + 41\right)\eta^3}{2160t_0^3(t_0 - 1)^3(2t_0 - 1)^3}+\bigO\left(\eta^4\right)-1\right)}\\[8pt]
&=& \dsp{- \frac{t_1^3\left(540t_0^4 - 1485t_0^3 + 1269t_0^2 - 369t_0 + 41\right)}{2160t_0^3(t_0 - 1)^3(2t_0 - 1)^3}+\bigO(\eta).}
\end{array}
\end{equation} 
We see a cancellation of terms in a way that we obtain a well defined value as $\eta\to 0$, and this will happen for all $g_{2k}$.  For this expansion of $g_2$ with the shown first term $g_{0,2}$, we need coefficients $t_1, t_2, t_3, t_4$ given in  \eqref{eq:largen30}.

\section*{Acknowledgments}

The authors thank the referees for their many relevant comments and suggestions, which helped us improve the presentation of our results. \\
The authors acknowledge financial support from project PID2021-127252NB-I00 funded by MCIN/AEI/10.13039/501100011033/ FEDER, UE. \\
NMT thanks CWI Amsterdam for scientific support.

\end{document}